\documentclass[12pt]{amsart} 

\usepackage[psamsfonts]{amssymb} 
\usepackage{amsfonts,amsmath}
\usepackage{color}

\usepackage{epsfig,multicol} 

\usepackage{xypic}

\input xy
\xyoption{all}

\newtheorem{thmA}{Theorem}

\newtheorem{corA}[thmA]{Corollary}

\newtheorem{theorem}{Theorem}[section]
\newtheorem{prop}[theorem]{Proposition}
\newtheorem{proposition}[theorem]{Proposition}
\newtheorem{lemma}[theorem]{Lemma}
\newtheorem{corollary}[theorem] {Corollary}
\newtheorem{addendum}[theorem] {Addendum}
\newtheorem{question}[theorem]{Question}

\theoremstyle{remark}
\newtheorem{remark}[theorem]{Remark}
\newtheorem{example}[theorem]{Example}

\theoremstyle{definition}
\newtheorem{definition}[theorem]{Definition}

\newcommand\redden[1]{\color{black}#1\color{black}}

\newcommand\mb[1]{\color{black}#1\color{black}}

\def\F{\mathcal F}
\def\Z{\mathbb Z}
\def\N{\mathbb N}
\def\Q{\mathbb Q}
\def\G{\Gamma}

\def\e{\varepsilon}
\def\SL{{\rm{SL}}}

\def\H{\Lambda}

\def\<{\langle}
\def\>{\rangle}

\title[Presenting finitely presentable groups]{On the difficulty of presenting finitely presentable groups}

\author[Martin R.\ Bridson]{Martin R.\ Bridson}
\address{Mathematical Institute, 24-29 St Giles', Oxford OX1 3LB, UK}
\email{bridson@maths.ox.ac.uk}

\author[Henry Wilton]{Henry Wilton}
\address{Mathematics 253-37, Caltech, Pasadena, CA 91125, USA}
\email{wilton@caltech.edu}

\thanks{Bridson is supported by a Senior
Fellowship from the EPSRC. Wilton is supported in part by NSF grant DMS-0906276.}

\date{minor edits 23/02/2011}


\begin{document}
\begin{abstract} 
We {\mb{exhibit}} classes of groups in which the word problem is uniformly solvable but in which there is no algorithm that can compute finite presentations for finitely presentable subgroups.  Direct products of hyperbolic groups, groups of integer matrices, and right-angled Coxeter groups form such classes. We discuss related classes of groups in which there {\em{does}} exist an algorithm to compute  finite presentations for finitely presentable subgroups.  We also construct a finitely presented group that has a polynomial Dehn function  but in which there is no algorithm to compute the first Betti number of its finitely presentable subgroups.
\end{abstract}
\maketitle

\centerline{\em{For Fritz Grunewald}}

\section{Introduction}

In the literature, the term ``finitely presented" is commonly used to describe a group that is isomorphic to a group of the form $F/N$ where $F$ is the free group on a finite set $A$ and $N$ is the subgroup generated by the conjugates of a finite set $R\subseteq F$; the casual notation $\G = \<A\mid R\>$ is often used in such circumstances. When greater precision is required, the more accurate term ``finitely presentable" is used to describe $\G$, and $\G$ is said to be {\em{finitely presented}} only when an explicit presentation $\G = \<A\mid R\>$ is given (i.e. a surjection $\pi:F\to\G$ and a choice of finite normal generating set
$R$ for the kernel of $\pi$). In this article we shall examine, from an algorithmic point of view, the content of the distinction between finite presentability and finite presentation in the context of subgroups of several well-behaved classes of groups, notably direct products of hyperbolic groups, groups of integer matrices, and right-angled Artin and Coxeter groups.

\smallskip

Consider the following statement: {\em{there exist finitely presented groups $G$ and integers $n$ such that there is no algorithm that, given a set of $n$ words in the generators of $G$  generating a finitely presentable subgroup $\H$, can calculate a presentation of $\H$; nor is there an algorithm that can calculate the 
dimension of $H_1(\H,\mathbb{Q})$.}}  After some reflection, the reader will realise that this is an immediate consequence of the fact that there exist finitely presented, torsion free groups with unsolvable word problem: take $G$ to be such a group and take $n=1$.
{\mb{More subtly, Collins \cite{collins}, building on work of McCool \cite{mccool}, established the existence of finitely
presented groups for which there is an algorithm to solve the word problem but no algorithm to determine the
order of an element. But what happens if we restrict our attention to classes of more geometrically significant groups
or classes that admit a uniform solution to the word problem, such as hyperbolic groups or residually finite groups?}}

\begin{definition}
One says that the {\em{finite presentation problem}} for a group $\G$ is {\em{solvable}} if there is an algorithm that, given  a finite subset $S\subseteq \G$ generating a finitely presentable subgroup $\Lambda$, outputs a finite presentation for $\Lambda$.

One says that the finite presentation problem for a class $\mathcal{C}$ of finitely presentable groups is \emph{uniformly solvable} if there is an algorithm that, given as input a  finite presentation for a group $\G\in\mathcal{C}$ and a finite subset $S\subseteq \G$ generating a finitely presentable subgroup $\Lambda$, will output a finite presentation for $\Lambda$.
\end{definition}

{\mb{It is important to note that this definition concerns the algorithmic construction of finite presentations
and not their mere existence; it should be contrasted with Definition \ref{d:def2}.}}
{\mb{It is also important to note that the}} algorithm is not required to give a correct answer or even to terminate if the input subset $S$ does not generate a finitely presentable subgroup.

Classes of groups that admit a uniform solution to the finite presentation problem include abelian groups, free groups, coherent right-angled Artin groups \cite{kapovich_foldings_2005}, locally quasiconvex hyperbolic groups, hyperbolic 3-manifold groups, certain Coxeter groups \cite{schupp_coxeter_2003}, and finitely presented residually free groups \cite{bridson_finitely_2008}.  In Section \ref{s: Examples} we shall discuss such positive results.  Our main results, though, concern  classes of groups, not far removed from those listed above, in which the uniform finite presentation problem is {\em{unsolvable}}, even though there is a uniform solution to the word problem.

For a fixed finite alphabet $A$, we shall denote by $A^{\pm}$ the set of letters $a\in A$ together with formal inverses $a^{-1}$, and by $A^{\pm*}$ the set of all finite words in the letters $A^{\pm}$.

\begin{thmA}\label{t:main}
There exists a finite set $Z$ and recursive sequences $(\Sigma_n),$ $ (S_n)$ of finite sets of words in $Z^{\pm *}$, of a fixed cardinality, such that:
\begin{enumerate}
\item for all $n\in\N$ the group $G_n:=\< Z\mid \Sigma_n\>$ is of the form
\[G_n
\cong\Gamma_n\times\Gamma_n
\]
where $\Gamma_n$ is a residually finite word-hyperbolic group of cohomological dimension 2;
\item for all $n\in\N$ the subgroup $\H_n\subseteq G_n$ generated by (the image of) $S_n$ is finitely presentable;
\item the set 
\[
\{n\mid b_1(\H_n)=b_1(G_n)\}
\]
is recursively enumerable but not recursive (where $b_1(\Lambda_n)$ denotes the first Betti number,  ${\rm{dim}}_\Q \,H_1(\Lambda_n,\Q)$).
\end{enumerate}
In particular, there does not exist an algorithm that, given $\< Z\mid \Sigma_n\> $ and $S_n$, can compute a 
finite presentation for $\H_n=\<S_n\>$.
\end{thmA}

\emph{Special} groups were defined by Haglund and Wise in \cite{haglund_special_2008}, and a group is \emph{virtually} special if it has a subgroup of finite index that is special.  The reader is referred to Section \ref{s: Matrix groups} for further details.  {\mb{Our next theorem improves the groups $\Gamma_n$ of Theorem \ref{t:main} to virtually special groups, but we lose the uniform bound on the size of the generating sets and the number of relations.}}

\begin{thmA}\label{t:main special}
There exists a recursive sequence of triples of finite sets $(Z_n,\Sigma_n,S_n)$  with $\Sigma_n,S_n\subseteq Z_n^{\pm *}$ such that:
\begin{enumerate}
\item for all $n\in\N$ the group $G_n:=\< Z_n\mid \Sigma_n\>$ is of the form
\[G_n
\cong\Gamma_n\times\Gamma_n
\]
where $\Gamma_n$ is word-hyperbolic, ${\rm{CAT}}(0)$\footnote{i.e. it acts properly and cocompactly on a complete
1-connected metric space that is non-positively curved in the sense of A.D.~Alexandrov \cite{bridson_metric_1999}.} and virtually special;
\item for all $n\in\N$ the subgroup $\H_n\subseteq G_n$ generated by (the image of) $S_n$ is finitely presentable;
\item the set 
\[
\{n\mid b_1(\H_n)=b_1(G_n)\}
\]
is recursively enumerable but not recursive.
\end{enumerate}
In particular, there does not exist an algorithm that, given $\< Z_n\mid \Sigma_n\> $ and $S_n$, can compute a 
finite presentation for $\H_n=\<S_n\>$.
\end{thmA}

Special groups are (virtually) subgroups of right-angled Artin groups and right-angled Coxeter groups.  By combining this fact with Theorem \ref{t:main special}, one obtains many classes of groups in which the finite presentation problem is not uniformly solvable; we list some in the following corollary.

\begin{corA}\label{c: Examples}
The finite presentation problem is not uniformly solvable in any of the following classes of groups:
\begin{enumerate}
\item direct products of hyperbolic groups;
\item ${\rm{CAT}}(0)$ groups;
\item\label{item: Z-linear} $\Z$-linear groups (i.e. groups of integer matrices);
\item right-angled Coxeter groups;
\item right-angled Artin groups.
\end{enumerate}
(Note that the word problem is uniformly solvable in each of the above classes of groups.)
\end{corA}

The pathological $\Z$-linear groups that lead to part (\ref{item: Z-linear}) of Corollary \ref{c: Examples} arise in the first instance as a sequence of presentations, together with generating sets for certain finitely presentable subgroups.  However,
the sharper arguments in Section 6 provide explicit faithful representations of these groups.

\begin{corA}\label{c: Matrices}
There exists a recursive sequence of triples $(m_n,r_n,S_n)$
with $m_n,r_n\in\N$ and  $S_n\subseteq \SL_{m_n}(\Z)$ a finite set, such that:
\begin{enumerate}
\item all of the groups $\Lambda_n:=\langle S_n\rangle\subseteq \SL_{m_n}(\Z)$ are finitely presentable;
\item the set of integers $\{n\in\N\mid b_1(\Lambda_n)=r_n\}$ is recursively enumerable but not recursive.
\end{enumerate}
In particular, there is no algorithm that takes as input a finite set of integer matrices that generate a finitely presentable group and outputs a presentation for that group.
\end{corA}

This provides our first example of a single group with unsolvable finite presentation problem: $\SL_{\infty}(\Z)$,  the direct limit of the `top left' inclusions $\SL_n(\Z)\hookrightarrow \SL_{n+1}(\Z)$.  Of course, this example is not finitely presentable; indeed it is not even finitely generated.   The latter defect can be remedied by applying a standard technique of Higman, Neummann and Neumann that embeds any countable group in a finitely generated group \cite{higman_embedding_1949}.  After that
{\mb{(as Collins \cite{collins} did),}} one could use a refinement of Higman's Embedding Theorem, due to Clapham \cite{clapham_embedding_1967}, to embed $\SL_\infty(\Z)$ in a finitely presented group with solvable word problem.  
{\mb{In Section \ref{s: Polynomial} we use  a more controlled embedding  due to Birget, Olshanskii, Rips and Sapir \cite{birget_isoperimetric_2002} to prove the following sharper result.}}

\begin{thmA}\label{t:single group}
There exists a finitely presented group $G$ that has a polynomial Dehn function but in which there is no algorithm to compute the first Betti number of finitely presentable subgroups.  In particular, the finite presentation problem is unsolvable in $G$.
\end{thmA}

The reader will recall that a group is termed {\em{coherent}} if all of its finitely generated subgroups are finitely presentable. Groves and Wilton \cite{groves_enumerating_2009} studied  \emph{effective coherence}, which for coherent groups is equivalent to having a solvable finite presentation problem. For groups that are not coherent, there is a natural companion to the finite presentation problem.

\begin{definition}\label{d:def2}
One says that the {\em{finite presentability problem}} for a group $\G$ is solvable if there is an algorithm that, given a finite subset $S\subseteq\G$, will determine whether or not the subgroup generated by $S$ is finitely presentable.

Similarly, one says that the finite presentability problem for a class $\mathcal{C}$ of finitely presentable groups is \emph{uniformly solvable} if there is an algorithm that, given a finite
presentation for a group $\G\in\mathcal{C}$ and a finite subset $S\subseteq \G$, will determine whether or not the subgroup generated by $S$ is finitely presentable.
\end{definition}

At first glance, it may seem improbable that there should exist reasonable groups  for which the finite presentability problem is unsolvable but the finite presentation problem is solvable.  But in fact there are many such groups, as we shall explain in Section \ref{s: Examples}. The simplest example is the direct product of two non-abelian free groups.

\mb{An underlying theme of this paper is the level of pathology to be found amongst the finitely presentable
subgroups of direct products of hyperbolic groups. There is a basic template for finding such pathologies: one begins with
a complicated finitely presented group, applies some form of the Rips construction to it, and then takes a fibre product
(see Subsections \ref{ss:rips} and \ref{ss:fib}). This scheme originates in \cite{baumslag_fibre_2000} and relies on the 1-2-3 Theorem proved there; see also \cite{bms}.}

This paper is organised as follows. In Section \ref{s2} we examine how the solvability of the (uniform) finite presentation problem is affected by passage to subgroups and overgroups of finite index. In Section \ref{s3} we gather from the literature such constructions and results as we need in our Main Construction. Section \ref{s:homology} contains homology calculations that are needed in the proofs of Theorem \ref{t:main} and Theorem \ref{t:main special}; {\mb{the arguments here are inspired by those used in
Section 3 of \cite{BMiller10} to prove a precursor of Theorem B.}} Section \ref{s: Proofs} contains the proofs themselves, as well as the proof of Corollary \ref{c: Examples}.  In Section \ref{s: Matrix groups} we explain how to compute algorithmically an explicit faithful representation of a virtually special group, and we deduce Corollary \ref{c: Matrices}.  Section \ref{s: Polynomial} contains a proof that any countable group with polynomial-time word problem can be embedded in a finitely generated group with polynomial-time word problem; from this we deduce Theorem \ref{t:single group}.  In Section \ref{s: Examples} we compare and contrast classes of groups that do and do not admit a uniform solution to the finite presentation problem. Section \ref{s:qus} contains a list of open questions.

\medskip

We thank Charles F.~Miller III for extensive and helpful comments on an earlier draft of this paper.

\section{Virtual Considerations}\label{s2}

Let $\mathcal{C}$ be a class of  groups, closed under isomorphism.  We shall prove that if $\mathcal{C}$ is recursively enumerable and admits a uniform solution to the finite presentation problem, then so (modulo a technicality)
do the class $\mathcal{SC}$  of all finitely presented subgroups of groups from  $\mathcal{C}$ and the class  
\[
\mathcal{VC}=\{G\mid \exists G_0\in\mathcal{C},\ G_0<G,~|G:G_0|<\infty\}
\]
of  `virtually $\mathcal{C}$ groups'.

First we consider how to present an extension of one finitely presented group by another.  For this purpose, it is convenient to adopt the following temporary notation: an arrow over a letter in a group presentation indicates an ordered set, e.g. $\vec{a}$, and the same letter with a subscript indicates a typical element of that set, e.g.~$a_i$; words are regarded as functions that take ordered sets as arguments, so $\vec{r}(\vec{a})$ is an ordered set of words $r_j(\vec{a})$ in the letters $a_i^{\pm 1}$; the notation $\vec{a}^{\vec{b}}$ represents the bi-ordered set of words of the form $a_i^{b_j}\equiv b_j^{-1}a_ib_j$.

\begin{lemma}\label{l: Presenting extensions}
Consider a short exact sequence
\[
1\to N\to G\to Q\to 1,
\]
with $N = \langle \vec{a}\mid \vec{r}(\vec{a})\rangle$ and $Q = \langle \vec{b}\mid \vec{s}(\vec{b})\rangle$.  For each $b_k$ in $\vec{b}$, let $\beta_k$ be an element of $G$ mapping to $b_k\in Q$. Noting that each $s_l(\vec{\beta})\in N$, select a word $\sigma_l(\vec{a})$ such that $\sigma_l(\vec{a})=s_l(\vec{\beta})$ in $G$.  Likewise, noting that $a_i^{\beta_k}\in N$ for all indices $i$ and $k$, select words $\alpha_{i,k}(\vec{a})$ such that $a_i^{\beta_k}=\alpha_{i,k}(\vec{a})$ in $G$. Then
\[
G = \langle \vec{a},\vec{b}\mid \vec{r}(\vec{a}),~\vec{\sigma}(\vec{a})=\vec{s}(\vec{b}),\ \vec{a}^{\vec{b}} = \vec{\alpha}(\vec{a})\rangle 
\]
\end{lemma}

\begin{proof}
Let $\widehat{G}$ be the group defined by the given presentation.  By
constuction, there is an epimorphism $\eta:\widehat{G}\to G$ defined by $a_i\mapsto a_i$ and $b_k\mapsto\beta_k$.  We must prove that $\eta$ is injective.  Let $\widehat{K}$ be the subgroup of $\widehat{G}$ generated by $\vec{a}$ (this is normal by construction) and let $\widehat{Q}$ be the group obtained from $\widehat{G}$ by setting each $a_i$ equal to $1$.  Then we have the following commutative diagram.\\
\centerline{
\xymatrix{
	1\ar@{>}[r] & \widehat{K}\ar@{>}[r]\ar@{>}[d] & \widehat{G}\ar@{>}[r]\ar@{>}[d]^{\eta} & \widehat{Q}\ar@{>}[r]\ar@{>}[d] & 1\\
	1\ar@{>}[r] & K\ar@{>}[r] & G\ar@{>}[r] & Q\ar@{>}[r] & 1\\  
}}
The map $\widehat{Q}\to Q$ is an isomorphism and the map $\widehat{K}\to K$ is injective, because every relation of $K$ also holds in $\widehat{K}$.  It follows that $\eta$ is injective as required.
\end{proof}

\begin{remark}
If one has a means of checking equalities in $G$, such as an explicit embedding into a recursively presented group, then a naive search will always find explicit sets of words $\vec{\sigma}$ and $\vec{\alpha}$, so the  construction of the above presentation becomes effective. 
\end{remark}

We remind the reader that a group $G$ is termed {\em{Hopfian}} if every
epimorphism $G\to G$ is an isomorphism, and {\em{locally Hopfian}} if
all finitely generated subgroups of $G$ are Hopfian. Residually
finite groups are locally Hopfian.

\begin{proposition}\label{p: Robustness}
If $\mathcal{C}$ is a recursively enumerable class of finitely presented, locally Hopfian groups and the finite presentation problem is uniformly solvable in $\mathcal{C}$, then the finite presentation problem is uniformly solvable in $\mathcal{SC}$ and in $\mathcal{VC}$.
\end{proposition}
\begin{proof}
The proof that the finite presentation problem is solvable in $\mathcal{SC}$
follows that of Lemma 1.4 in \cite{groves_enumerating_2009}; it is straightforward, so we
omit the details. However, we draw the reader's attention to the
following subtlety that was overlooked in \cite{groves_enumerating_2009}. In the course of
the proof one reaches a stage where one has a finite subset
$S\subset G\in\mathcal{SC}$ and
a finite presentation $\<A\mid R\>$ of $\<S\>\subset G$, and one wants
to express the elements of $S$ as words in the generators $A$. To do
this, one can employ a naive search for homomorphisms  $\<A\mid R\>\to
\<S\>\subset  G$, choosing words in $S^{\pm 1}$ as putative images for
each $a\in A$ and checking that the defining relations hold in $ G$. A
further naive search will find one of these homomorphisms that is
surjective: working in parallel on all of the homomorphisms found, one
looks for words in $A^{\pm 1}$ mapping to each $s\in S$. Because $G$ is
locally Hopfian, the surjective map that one eventually finds is an
isomorphism.

Suppose that we are given a finite presentation for a group $G\in\mathcal{VC}$ and a finite set $S$ of words in the generators that generate a finitely presentable subgroup. We must construct a finite presentation for $\H=\<S\>$. 

The Reidemeister--Schreier Process can be used to enumerate presentations of subgroups of $G$ of finite index.  Recursively enumerating all possible presentations of groups in $\mathcal{C}$ using Tietze transformations and comparing them to presentations of finite-index subgroups of $G$, we will eventually find a finite-index subgroup $G_0\in\mathcal{C}$ and a finite presentation for $G_0$. 

A further application of the Reidemeister--Schreier Process will eventually find a finite presentation for a finite-index subgroup $G_1$ of $G_0$ that is normal in $G$.  Let $Q=G/G_1$ be the finite quotient group.

One can compute a finite set  $T$ of generators for $\H_1=G_1\cap \H$, as a set of words in $F(S)$ using Stallings's
method (see for example \cite{stallings_topology_1983}).  The hypothesised solution to the
finite presentation problem in $\mathcal{C}$ enables us to compute a finite presentation $\langle T\mid R\rangle$ for $\H_1$, and it is straightforward to compute a presentation $\langle S\mid U\rangle$ for $B=\H/\H_1\subseteq Q$.  The result now follows by Lemma \ref{l: Presenting extensions}.
\end{proof}

\section{The ingredients of the main construction}\label{s3}

\subsection{Collins--Miller groups} 

The following construction shows that there is no algorithm to determine if a finite group-presentation is aspherical and no algorithm that computes the second homology of the group (there is an earlier proof of this second fact due to Cameron Gordon \cite{gordon_embedding_1995}, cf. \cite{bogley_homological_2002}).  

\begin{theorem}[Collins--Miller \cite{collins_word_1999}]\label{t: Collins--Miller}
There exists a finite set $X$ and a recursive sequence of finite subsets $R_n \subseteq X^{\pm *}$,  of a fixed cardinality
greater than $X$, such that:
\begin{enumerate}
\item the group $Q_n = \<X \mid R_n\>$ is either trivial or else the given presentation is aspherical;
\item the set $\{n\in\N\mid Q_n\cong 1\}$ is recursively enumerable but not recursive.
\end{enumerate}
\end{theorem}

As there is a simple algorithm to compute $H_1(Q_n,\Z)$, the subsequence consisting of those $Q_n$ which are perfect is recursive; we pass to this subsequence.  Also, since the presentation of $Q_n$ is aspherical when $Q_n\not\cong 1$, we can use its standard 2-complex to compute the homology of $Q_n$. In particular $H_2(Q,\Z)$ is the kernel of the map from the cellular 2-chains to the cellular 1-chains.  Since $|R_n|>|X|$, there are more 2-cells than 1-cells, so this kernel is infinite.

\begin{addendum}\label{a:CM} Further,
\begin{enumerate}
\item[(3)] each of the groups $Q_n = \<X \mid R_n\>$ is perfect, and
\item[(4)] $H_2(Q_n,\Z)$ is infinite if $Q_n\not\cong 1$. 
\end{enumerate}
\end{addendum} 

\subsection{Universal central extensions}

We remind the reader that a central extension of a group $Q$ is a group $\widetilde Q$ equipped with a  homomorphism $\pi:\widetilde Q\to Q$ whose kernel is central in $\widetilde Q$. Such an extension is universal if, given any other central extension $\pi':E\to Q$, there is a unique homomorphism $f:\widetilde Q\to E$ such that $\pi'\circ f =\pi$.

The standard reference for universal central extensions is  \cite{milnor_introduction_1971}, pp. 43--47.  The homological properties that we need are summarized in the following proposition.

\begin{prop}\label{p: Properties of UCEs}
If $Q$ is a perfect group then:
\begin{enumerate}
\item $Q$ has a universal central extension $\widetilde{Q}$ (and, conversely, the existence of a universal central extension implies that $Q$ is perfect);
\item there is a short exact sequence
\[
1\to H_2(Q;\Z)\to \widetilde{Q}\overset{\pi}\to Q\to 1;
\]
\item $\widetilde{Q}$ is super-perfect, i.e. $H_2(\widetilde{Q},\Z)\cong H_1(\widetilde{Q},\Z)\cong 0$.
\end{enumerate}
\end{prop}

We also need to know that it is possible to write down a presentation for the universal central extension of a perfect group.  The following result is Corollary 3.6 of \cite{bridson_decision_2010}.

\begin{prop}\label{p: Presentation for universal central extension}
Let $Q= \langle x_1,\dots,x_n \mid r_1,\dots,r_m \rangle$ be a finitely presented, perfect group, let $F$ be the free group on $\{x_1,\dots,x_n\}$ and let $R$ be the normal closure of $\{r_1,\dots,r_m\}$ in $F$. Choose $c_i\in [F,F]$   such that $x_ic_i\in R$. Then
\[
\<x_1,\dots,x_n \mid x_ic_i,\ [x_i,\,r_j] \ (i=1,\dots,n; \ j=1,\dots,m) \>
\]
is a finite presentation for $\widetilde{Q}$, the universal central extension of $Q$.
\end{prop}

A naive search identifies suitable choices for the $c_i$, and hence the process of
passing from a finite presentation of a perfect group to a finite presentation of its universal
central extension is entirely algorithmic. 
Also, the image of $\{r_1,\ldots,r_m\}$ in $\widetilde{Q}$ generates $H_2(Q,\Z)\subseteq \widetilde{Q}$. 

\begin{addendum}\label{a:H_2} There is an algorithm that, given a finite
presentation of a perfect group $Q$, will construct a finite presentation for the
universal central extension $\widetilde Q$ together with a finite
generating set for $H_2(Q,\Z)\subseteq \widetilde{Q}$,
given as words in the generators of $\tilde Q$.
\end{addendum}

\subsection{The Rips Construction}\label{ss:rips} {\mb{The Rips  construction \cite{rips_subgroups_1982} is a remarkable tool for constructing hyperbolic groups with pathological behaviour.  It is extremely flexible and as a result one can insist that the hyperbolic
groups constructed enjoy various additional properties.   The following residually finite version is due to Wise.

\begin{theorem}[Wise \cite{wise_residually_2003}]\label{t: Rips}
There is an algorithm that, given a finitely presented group $Q=\langle X\mid R\rangle$, will construct a finite presentation $\< a_1,a_2,a_3,X\mid \Sigma\>$ of a residually finite hyperbolic group $\Gamma$ of cohomological dimension 2 and a short exact sequence
\[
1\to K\to\Gamma\overset{\eta}\to Q\to 1
\]
where $K\subseteq \G$ is the subgroup generated by $\{a_1,a_2,a_3\}$ and $\eta:\G\to Q$ is defined by the identity map on $X$. Moreover $|\Sigma|= |R| + 6|X|$.
\end{theorem}

The following `virtually special' version is due to Haglund and Wise.
Note that here we have lost control over the rank of $K$.

\begin{theorem}[Haglund--Wise \cite{haglund_special_2008}]\label{t:HW}
There is an algorithm that, given a finitely presented group $Q=\<X\mid R\>$ will construct a finite presentation $\< A\cup X\mid \Sigma\>$ of a virtually special, ${\rm{CAT}}(0)$, hyperbolic group $\Gamma$ and a short exact sequence
\[
1\to K\to\Gamma\overset{\eta}\to Q\to 1
\]
where $K\subseteq\G$ is the group generated by $A$ and $\eta:\G\to Q$ is defined by the identity map on $X$.
\end{theorem}

{\mb{Virtually special groups enjoy many useful properties.
For example, because virtually special groups are linear over $\Z$
(see   Section \ref{s: Matrix groups}), Theorem \ref{t:HW}
and the Main Theorem of Bridson and Miller \cite{BMiller-decide} 
together imply:


\begin{theorem} If $n$ is sufficiently large, then the isomorphism problem 
for
finitely presented subgroups of 
$\SL(n,\Z)$
is unsolvable.
\end{theorem}
}}

\subsection{Fibre products}\label{ss:fib}

Given a short exact sequence of groups
\[
1\to K\to\Gamma\stackrel{\eta}{\to} Q\to 1
\]
the corresponding \emph{fibre product} is the subgroup $\H$ of $\Gamma\times\Gamma$ defined by
\[
\H=\{(\gamma_1,\gamma_2)\in\Gamma\times\Gamma\mid\eta(\gamma_1)=\eta(\gamma_2)\}.
\]
Projection onto the second factor gives an epimorphism $\H\to\Gamma$ with kernel $K\times 1$. The map that identifies $\Gamma$ with the diagonal subgroup of $\Gamma\times\Gamma$  splits this projection and therefore the fibre product can be expressed as a semidirect product
\begin{equation}\label{eq: Semidirect product}
\H\cong K\rtimes\Gamma,
\end{equation}
where the action of $\Gamma$ is the conjugation action in the original short exact sequence.

Recall that a group is \emph{of type $\F_n$} if it admits an Eilenberg--Mac~Lane space with finite $n$ skeleton.  
It is a simple exercise to prove that if $\G$ is finitely generated (i.e. of type $\F_1$) and $Q$ is finitely presentable (i.e. of type $\F_2$) then $\H$ is finitely generated. This easy `0-1-2 Lemma' has a more sophisticated analogue due to Baumslag, Bridson, Miller and Short \cite{baumslag_fibre_2000}.

\begin{theorem}[The 1-2-3 Theorem \cite{baumslag_fibre_2000}]
Let
\[
1\to K\to\Gamma\to Q\to 1
\]
be a short exact sequence of groups.  If $K$ is finitely generated
(type $\F_1$), $\Gamma$ is finitely presented (type $\F_2$), and $Q$ is of type $\F_3$, then the corresponding fibre product $\H\subseteq\Gamma\times\Gamma$ is finitely presentable.
\end{theorem}

The details of how one proves this theorem need not concern us here, but we do need the
following observation.

\begin{lemma}\label{l:size}
Regardless of whether $Q$ is of type $\F_3$, if $A$ generates $K$ and $X\cup A$ generates $\Gamma$, then $\H$ is generated by $\{(a,1),\, (1,a)\mid a\in A\}\cup \{(x,x)\mid x\in X\}$.
\end{lemma}

\begin{remark}
There is an `Effective 1-2-3 Theorem' (Theorem 2.2 of \cite{bridson_finitely_2008}), which gives an algorithm that will compute a finite presentation for $\H$. This algorithm
takes as input finite presentations for $\Gamma$ and $Q$ and a finite set of $\Z Q$-module generators for the second homotopy group of a presentation complex for $Q$.  It follows from our proof of 
Theorem \ref{t:main} that the input for the Effective 1-2-3 Theorem cannot be reduced to the finite presentations for $\Gamma$ and $Q$ and the abstract knowledge that $Q$ is of type $\F_3$.
\end{remark}

\section{The rational homology of fibre products}\label{s:homology}

{\mb{This section is based on Section 3 of \cite{BMiller10}.}}
All homology groups in this section, unless otherwise stated, are with trivial coefficient module $\mathbb Q$ (which is omitted from the notation). The $i$th Betti number of a group $G$, denoted $b_i(G)$, is the dimension of $H_i(G)$ as a vector space over $\mathbb{Q}$.  Let
\[
1\to K\to \G\stackrel{\eta}{\to} Q\to 1
\]
be a short exact sequence of groups, defining a fibre product $\H\subseteq\G\times\G$ as above.  We will relate $b_1(\H)$ to  $b_2(Q)$ and $b_1(\Gamma)$, proceeding under the following assumptions:
\begin{enumerate}
\item[(i)] $b_1(Q)=0$;
\item[(ii)] the map $H_2(\Gamma)\to H_2(Q)$ induced by $\eta$ is zero.
\end{enumerate}

The short exact sequence gives rise to the following five-term exact sequence in homology \cite{brown_cohomology_1994}:
\[
H_2(\G) \stackrel{\eta_*}{\to} H_2(Q)\to H_0(Q,H_1(K))\to H_1(\G) \stackrel{\eta_*}{\to} H_1(Q)\to 0.
\]

The first arrow is induced by $\eta$, so under our hypotheses we obtain a short exact sequence of rational vector spaces and deduce that
\begin{equation}\label{eq: Betti number}
b_2(Q)+b_1(\G)=\dim_\mathbb{Q} H_0(Q,H_1(K)).
\end{equation}

\begin{lemma}
If $\H\subseteq\G\times\G$ is the fibre product associated to the short exact sequence
\[
1\to K\to\Gamma\to Q\to 1
\]
then $H_1(\H) \cong H_0(Q,H_1(K))\oplus H_1(\G)$.
\end{lemma}
\begin{proof}
The group $\G$ acts on $K$ by conjugation, inducing an action of $Q=\G/K$ on $H_1(K)$. By definition, $H_0(Q,H_1(K))$ is the quotient of $H_1(K)$ by this action. 
As in equation (\ref{eq: Semidirect product}),  we have $\H\cong K\rtimes\G$. Thus $H_1(\H)$ is the sum of $H_1(\G)$, the abelianisation of $\G$, and the quotient of $H_1(K)$ obtained by trivialising the $\G$ action, that is $H_0(Q,H_1(K))$.  Thus the decomposition of $H_1(\H)$ is an immediate consequence of equation (\ref{eq: Semidirect product}).\end{proof}

Combining this lemma with equation (\ref{eq: Betti number}), we have proved the following formula.

\begin{prop}\label{p: Rational homology formula}
Suppose $1\to K\to\G\to Q\to 1$ is a short exact sequence of groups, where $Q$ is perfect and the induced map $H_2(\G)\to H_2(Q)$ is zero. Then for the corresponding fibre product $\H$ we have
\[
b_1(\H)=2b_1(\G)+b_2(Q).
\]
\end{prop}

Thus if one is in a situation where one can calculate $b_1(\G)$ but not $b_2(Q)$, then one cannot calculate  $b_1(\H)$.

\section{Proofs} \label{s: Proofs}

We shall deduce Theorem \ref{t:main} from the following proposition, the proof of which is a straightforward concatenation of the results of the previous sections.

\begin{prop}\label{p: Technical proposition}
There exists an algorithm that takes as input a finite presentation $\langle X\mid R\rangle$ for a perfect group $Q$ and outputs a finite presentation $\langle Y\mid T\rangle$ for a residually finite hyperbolic group $\Gamma$ as well as a finite set $S\subseteq Y^{\pm*}\times Y^{\pm*}$ with the property that the subgroup $\H$ of $\Gamma\times\Gamma$ generated by $S$ satisfies
\[
b_1(\H)=2b_1(\Gamma)+b_2(Q).
\]
Moreover, $|Y|$, $|S|$ and $|T|$ depend only on $|X|$ and $|R|$.
\end{prop}
\begin{proof}
Given $\<X\mid R\>$ we first construct a finite presentation $\<X\mid \widetilde{R}\>$ for the universal central extension of $Q$, as in Addendum \ref{a:H_2}:
\[
1\to J\to\widetilde{Q}\stackrel{\xi}{\to} Q\to 1
\]
where $J\cong H_2(Q,\Z)$ is generated by the specific set of words $R(X)\subseteq X^{\pm*}$. We then apply Theorem \ref{t: Rips} to $\<X\mid \widetilde{R}\>$ to obtain a  short exact sequence
\[
1\to K\to\Gamma\stackrel{\eta}{\to} \widetilde{Q}\to 1
\]
where $\Gamma=\<Y \mid T\>$ is a hyperbolic group generated by $Y=\{a_1,a_2,a_3\}\cup X$ and $K$ is the subgroup generated by $\{a_1,a_2,a_3\}$.  Let $L=\eta^{-1}(J)$ and note that this is the normal subgroup of $\Gamma$ generated (as a subgroup) by the finite set $\{a_1,a_2,a_3\}\cup R(X)$. We have a short exact sequence
\[
1\to L\to \Gamma\stackrel{\zeta}{\to} Q\to 1  
\]
where $\zeta=\xi\circ\eta$ is induced by the identity map on $X$.  The fibre product $\H\subseteq\G\times\G$ associated to this short exact sequence has a finite generating set $S$ as described in Lemma \ref{l:size}.

Every part of the construction up to this point is explicit: we have  described an algorithm that yields a presentation $\langle Y\mid T\rangle$ for $\Gamma$ and a finite set $S$ of generators for $\H$. Moreover, because of the explicit nature of the presentations in Proposition \ref{p: Properties of UCEs}  and Theorem \ref{t: Rips}, we see that $|Y|$, $|S|$ and $|T|$ depend only on the number of generators and relations in our original presentation of $Q$.

By assumption, $Q$ is perfect.  The map $\zeta$ factors through $\widetilde{Q}$; as $H_2(\widetilde{Q})=0$ it follows that $\zeta$ is zero at the level of second homology.  Therefore, $\zeta$ satisfies the hypotheses of Proposition \ref{p: Rational homology formula} and the result follows.
\end{proof}

To prove Theorem \ref{t:main}, we apply Proposition \ref{p: Technical proposition} to the sequence of presentations produced by the Collins--Miller construction.  Finite presentability comes from the 1-2-3 Theorem.

\begin{proof}[Proof of Theorem \ref{t:main}]
Let $\<X\mid R_n\>$ be the sequence of presentations provided by the Collins--Miller construction, modified as in Addendum \ref{a:CM}.  For each $n$, the algorithm of Proposition \ref{p: Technical proposition} applied to $\langle X\mid R_n\rangle$ yields a residually finite hyperbolic group $\Gamma_n$ with an explicit finite presentation $\<Y_n\mid T_n\>$.  Note $Y_n=\{a_1,a_2,a_3\}\cup X$, which we shall refer to simply as $Y$, for all $n$.  From this presentation we derive a presentation $\<Z\mid \Sigma_n\>$ for $\Gamma_n\times\Gamma_n$ in the obvious manner.  Proposition \ref{p: Technical proposition} also yields a subgroup $\H_n\subseteq\Gamma_n\times\Gamma_n$ given by a finite set of words $S_n\subseteq Z^{\pm*}$.  The cardinalities of $\Sigma_n$ and $S_n$ depend only on $|X|$ and $|R_n|$ and therefore are independent of $n$, and Proposition \ref{p: Technical proposition} tells us that $b_1(\H)=2b_1(\Gamma)+b_2(Q).$ 

The first Betti number of $\Gamma_n$ is readily computable from its presentation, whereas $b_2(Q_n)=0$ if and only if $Q_n$ is the trivial group, by Addendum \ref{a:CM}.  Therefore the set of $n$ such that $b_1(\H_n)=2b_1(\Gamma)$ is not recursive, by part (2) of Theorem \ref{t: Collins--Miller}.

Each of the groups $Q_n = \<X\mid R_n\>$ is either trivial or else has an aspherical presentation. In either case, $Q_n$ is of type $\F_3$. Thus, by the 1-2-3 Theorem, each $\H_n$ is finitely presentable.
\end{proof}

\begin{remark}
The careful reader may be concerned that the sequence of presentations $G_n=\langle Z\mid \Sigma_n\rangle$ that we constructed in the proof of Theorem \ref{t:main} is \emph{a priori} recursively enumerable rather than recursive.  But there is a simple device that overcomes such concerns (cf.~\cite{cfm-thesis}, p.2): one can transform any recursively enumerable sequence of presentations $G_n=\langle Z_n\mid \Sigma_n\rangle$ into a recursive sequence simply by adding an extra generator $\zeta$ to the generating sets $Z_n$ and adding the relations $\zeta^n=1$ and $\zeta=1$ to $\Sigma_n$; this does not alter the isomorphism type of the group $G_n$, but it is easy to see that the new sequence is recursive.
\end{remark}

The proof of Theorem \ref{t:main special} is entirely similar to that of Theorem \ref{t:main}, except that we apply Theorem \ref{t:HW} instead of Theorem \ref{t: Rips} to obtain the following analogue of Proposition \ref{p: Technical proposition}.

\begin{prop}\label{p: Technical proposition special}
There exists an algorithm that takes as input a finite presentation $\langle X\mid R\rangle$ for a perfect group $Q$ and outputs a finite presentation $\langle Y\mid T\rangle$ for a virtually special, CAT(0), hyperbolic group $\Gamma$ as well as a finite set $S\subseteq Y^{\pm*}\times Y^{\pm*}$ with the property that the subgroup $\H$ of $\Gamma\times\Gamma$ generated by $S$ satisfies
\[
b_1(\H)=2b_1(\Gamma)+b_2(Q).
\]
\end{prop}

Finally, we explain how to deduce Corollary \ref{c: Examples} from Theorem \ref{t:main special}.  In the following proof, we assume that the reader is familiar with the terminology of virtually special, Coxeter and Artin groups, as recalled in the next section.

\begin{proof}[Proof of Corollary \ref{c: Examples}]
The sequence of groups provided by Theorem \ref{t:main special} are all direct products of hyperbolic groups, so part (1) follows.  By construction, the groups provided by Theorem \ref{t:main special} are CAT(0), so part (2) follows.  Virtually special groups are $\Z$-linear by Theorem 1.1  of \cite{haglund_special_2008}; this proves part (3).  By Theorem 4.2 of \cite{haglund_special_2008},  the groups $G_n$ from Theorem \ref{t:main special} are virtually subgroups of right-angled Coxeter groups.  Note that the class of right-angled Coxeter groups is recursively enumerable.  By Proposition \ref{p: Robustness}, if there were a uniform solution to the finite presentation problem in the class of right-angled Coxeter groups then there would be a uniform solution in the class of virtual subgroups of right-angled Coxeter groups, contradicting Theorem \ref{t:main special}.  This proves part (4).  Part (5) is similar: the class of right-angled Artin groups is recursively enumerable, and the groups $G_n$ are virtually subgroups of right-angled Artin groups, again by Theorem 4.2 of \cite{haglund_special_2008}.
\end{proof}

\section{Matrix groups}\label{s: Matrix groups}

In order to deduce Corollary \ref{c: Matrices} from part (3) of Corollary \ref{c: Examples}, we need an algorithm that will construct an explicit faithful representation over $\Z$ of a virtually special group.

Roughly speaking, a \emph{cubical complex} is a CW-complex in which the $k$-cells are Euclidean $k$-cubes attached by local isometries along $(k-1)$-dimensional faces (see page 112 of \cite{bridson_metric_1999}).   We will only be concerned with compact cubical complexes.  Such a complex is \emph{non-positively curved} if it admits a metric that is locally CAT(0).  In \cite{haglund_special_2008}, the subclasses of \emph{special} cube complexes, \emph{A-special} cube complexes and  \emph{C-special} cube complexes were introduced; the reader is referred to that paper for their definitions.  A group is called \emph{special} (or \emph{C-special}, or \emph{A-special}) if it is the fundamental group of the corresponding sort of cubical complex.  As usual, a group $G$ is \emph{virtually special} if it has a finite-index subgroup that is special.  Proposition 3.10 of \cite{haglund_special_2008} shows that a virtually special group is virtually C-special and virtually A-special.     Recall that a \emph{right-angled Artin group} $A_N$, defined by a finite graph $N$, has a presentation in which the generators are the vertices of $N$ and two generators commute if and only if they are joined by an edge in $N$.  A \emph{right-angled Coxeter group} $C_N$ is defined similarly, with the additional stipulation that each generator is an involution.  We are most interested in the following theorem of \cite{haglund_special_2008}.

\begin{theorem}
If $G$ is C-special then $G$ embeds into a right-angled Coxeter group.  If $G$ is A-special then $G$ embeds into a right-angled Artin group.
\end{theorem}

As right-angled Coxeter groups are $\Z$-linear, it follows that a virtually special group is $\Z$-linear.  In this section, we will show how to find an explicit representation over $\Z$ algorithmically.  Our algorithms will work with combinatorial descriptions of cubical complexes. To simplify the exposition, we will not specify a particular way of representing cubical complexes combinatorially, but it is clear that such descriptions exist.  

\begin{lemma}\label{l: Link condition}
There is an algorithm that takes as input a combinatorial description for a finite cubical complex $X$ and determines whether or not $X$ is non-positively curved.
\end{lemma}

This is an immediate consequence of Gromov's Link Condition, which asserts that a cubical complex is non-positively curved if and only if the link of each vertex is flag (\cite{bridson_metric_1999}, Theorem II.5.20).

\begin{lemma}\label{l: C-special recognition}
There is an algorithm that takes as input a combinatorial description for a finite cubical complex $X$ and determines whether or not $X$ is C-special.  
\end{lemma}

This follows from the definition of C-special (Definition 3.2 in \cite{haglund_special_2008}).  The next lemma is a direct consequence of the results of \cite{haglund_special_2008} (see in particular Definition 3.14 and Theorem 4.2).

\begin{lemma}\label{l: Coxeter embedding}
There is an algorithm that takes as input a combinatorial description for a finite, non-positively curved, C-special cube complex $X$ and outputs a finite graph $N$ and an injective homomorphism $\pi_1(X)\to C_N$.
\end{lemma}

\begin{lemma}\label{l: Special representation}
There is an algorithm that takes as input a finite presentation $\langle X\mid R\rangle$ for a virtually special group $G$ and outputs an integer $m$, a finite set of matrices $\Xi\subseteq \SL_m(\Z)$ and a bijection $X\to \Xi$ that defines an injective homomorphism $G\to \SL_m(\Z)$.
\end{lemma}
\begin{proof}
The Reidemeister--Schreier process enumerates presentations for the finite-index subgroups $H_i$ of $G$.  As one can recursively enumerate finite cubical complexes, it follows from Lemmas \ref{l: Link condition} and \ref{l: C-special recognition} that one can enumerate finite, non-positively curved, C-special cube complexes. A naive search will eventually find an isomorphism between a finite-index subgroup $H_i$ of $G$ and the fundamental group of a finite, C-special cube complex $X$.  Composing this with the homomorphism provided by Lemma \ref{l: Coxeter embedding} gives an injective homomorphism $H_i\hookrightarrow C_N$.  One can write down a faithful representation $C_N\to \SL_d(\Z)$ (see, for instance, pages 109--10 of \cite{humphreys_reflection_1990}), and composing these maps gives a faithful representation $\rho:H_i\to \SL_d(\Z)$.  Finally, the induced representation $\mathrm{Ind}^G_{H_i}\rho$, which is easy to compute, is a faithful homomorphism from $G$ to $\SL_m(\Z)$, where $m=d|G:H_i|$.
\end{proof}

Corollary \ref{c: Matrices} follows immediately from Theorem \ref{t:main special} and Lemma \ref{l: Special representation}, taking $r_n=b_1(G_n)$.

\section{Polynomial-time solutions to the word problem}\label{s: Polynomial}

In this section we prove Theorem \ref{t:single group}. Our proof
requires us to be clear about what it means to have a polynomial time
solution to the word problem in a countable group $\G$ with infinite
generating set $\{c_0,c_1,c_2,\ldots\}$. To this end, we use the
map $\nu: c_n^{\pm 1}\mapsto b^na^{\pm 1}b^{-n}$ to embed the free
monoid $C^*$ on the infinite set $C=\{c_n,c_n^{-1}\mid n\in \N\}$
in the free monoid on the finite set
$A=\{a,b,a^{-1},b^{-1}\}$. We say that  the word problem in $\Gamma$
is {\em solvable in polynomial time} (resp.~is in NP) if there is a
deterministic (resp.~non-deterministic) Turing
machine with input alphabet $A$ that accepts exactly  the language
$\{w=\nu(u) \mid
u\in C^*,\ u=1 \text{ in }\G\}$ and which has a polynomially-bounded
time function. This technical device captures the intuitive idea of
a Turing machine that solves the word problem by working with the infinite alphabet $c_0,c_1,c_2,\ldots$ and which
given a word $w=c_{i_1}^{\epsilon_1}\ldots c_{i_l}^{\epsilon_l}$ 
with $\e_i=\pm 1$, will halt and decide if $w=1$ in $\G$ 
in at most $q(l+\mu)$ steps, where $q$ is a fixed polynomial and
$\mu=\max\{i_1,\dots,i_l\}$.  

The following result and its proof will strike experts as routine, but we
were unable to find it in the literature and we need it in
our proof of Theorem \ref{t:single group}.}}

\begin{theorem}\label{t: Polynomial word problem}
Let $\Gamma$ be a countable
group
with an infinite generating set $\{c_0,c_1,\ldots \}$
such that the word problem is solvable in polynomial {\mb{(resp.~NP)}} time. Then $\Gamma$ can be embedded in a \redden{four}-generator group $G$ whose word problem is solvable in polynomial {\mb{(resp.~NP)}} time.
\end{theorem}

Combining this with {\mb{the celebrated theorem of Birget, Olshanskii, Rips
and Sapir, Theorem 1.1 of \cite{birget_isoperimetric_2002},}} we get:

\begin{corollary}\label{c: Polynomial Dehn function}
If $\Gamma$ is a
countable
group
with an infinite generating set $\{c_0,c_1,\ldots \}$
such that the word problem is solvable in NP time, then $\Gamma$ can be embedded in a finitely presented group with polynomial Dehn function.
\end{corollary}

We construct $G$ according to the following well known prescription, essentially due to Higman, Neumann and Neumann \cite{higman_embedding_1949}. 
For each $n\geq 0$ the subgroup $H_n$ of $\Gamma\redden{*\langle x\rangle}*\langle s\rangle$ generated by the set
\[
S_n=\{s_i=s^{i+1}c_i\redden{x}s^{-i-1}\mid i\geq n\}
\]
is freely generated by $S_n$.  Therefore, the map that sends $s_i$ to $s_{i+1}$ for each $i$ extends to an isomorphism $H_0 \to H_1$, which defines an HNN extension $G$ with stable letter $t$; so $ts_it^{-1}=s_{i+1}$ for each $i$.   Note that $G$ is generated by $s$, $t$\redden{, $x$ }and $c_0$.

Let $F_0$ be the free group on $\{s,t,\redden{x,}c_0\}$ and let $\eta:F_0\to G$ be the natural epimorphism.  To prove the theorem, we will describe an algorithm (resp. non-deterministic algorithm), taking as input words $w_0\in F_0$,
running in polynomial time (as a function of $|w_0|$) and determining
whether or not $\eta(w_0)=1$.  

Basically, the algorithm repeatedly looks for a pinch $tut^{-1}$ with $\eta(u)\in H_0$
(or $t^{-1}vt$ with $\eta(v)\in H_1$) and removes it, thus reducing the
number of occurences of $t$ in the given word $w_0$. One would like to say
that
this search and removal
can be done in polynomial (resp. NP) time using the solution to the word problem
in $\G$. But in order to make this sort of argument one needs to translate
$u$ into a form where the solution to the word problem can be applied.
Thus it will be convenient to work with the free group
$F_\infty$ that is the direct limit of the free groups $F_n
={\rm{Free}} \{s,t,\redden{x},c_0,\ldots,c_n\}$.  

The epimorphism $\eta$ extends to the natural map (also denoted $\eta$) from $F_\infty$ onto $G$. We write $l(u)$ for the reduced length of elements $u\in F_\infty$.
As in the opening paragraph of this section, we want a finite encoding
of $F_\infty$, so we embed it in $\widehat{F}$, the free group on $\{s,t,x,a,b\}$, via the map $\nu(s)=s$, $\nu(t)= t$, \redden{$\nu(x)=x$ }and $\nu(c_n)= b^nab^{-n}$.  Let $\hat{l}(w)$ denote the reduced
length of $\nu(w)$.  

This encoding is efficient in the following two senses.

\begin{lemma}
If $u\in F_n$ then
\[
l(u)\leq \hat{l}(u)\leq (2n+1)l(u)\ .
\]
{\mb{And there}} is an algorithm, polynomial-time in the length of $W$,
that given a {\mb{word $W$
in the letters}} $\{s,x, a,b\}$, will determine if $W\in\nu(F_\infty)$,
and  will output\footnote{the output 
strings should be written in a finite alphabet, so products $c_{i_1}c_{i_2}\dots$ should be specified by strings of integers (written in binary perhaps) $i_1,i_2,\dots$}
the reduced word $w\in F_\infty$ with $\nu(w)=W$ (if it
exists).
\end{lemma}

\begin{proof} The first claim is obvious. The algorithm in the second claim
first writes   $W$ in reduced form
\[
W=V_0s^{i_0}V_1s^{i_1}\ldots V_ks^{i_k}
\]
then writes each $V_j=V_j(x,a,b)$ as a reduced product of powers of 
$x$ and reduced words $U_l=U_l(a,b)$. Now, $W$ is in $\nu(F_\infty)$
if and only if all of the $U_l$ are. For each $U_l$,  let $a^{\epsilon_i}$ be the $i$th occurrence of $a^{\pm 1}$, let $P_i$ be the prefix of $U_l$ that ends with $a^{\epsilon_i}$, and let $\beta_i$ be the exponent sum of $b$ in $P_i$.  Then $U_l\in\nu(F_\infty)$ if and only if
\[
U_l^{-1} \ \prod_{i} (b^{\beta_i}a^{\epsilon_i}b^{-\beta_i})  = 1,
\]
which is easily checked. 
If this product is trivial then $U_l=\nu(u_l)$ where
\[
u_l=\prod_{i}c_{\beta_i}^{\epsilon_i}.
\] 
Replacing each $U_l$ in the reduced form of $W$ by this expression for $u_l$
gives a reduced word $w\in F_\infty$ with $\nu(w)=W$.
\end{proof}

In order to  recognise pinches and hence
solve the word problem in $G$,
we have to be able to recognise words $W=\nu(w)$ 
with $\eta(w)\in H_i$ for $i=0,1$. To do so we employ the solution to the
word problem in $\G$, which we are assuming is either polynomial-time or $NP$.

\begin{lemma}\label{l2} There is a polynomial time (resp. NP) algorithm that, given
a word $W$ in the free group on $\{s,x,a,b\}$, will determine whether or not
$W=\nu(w)$ with 
\begin{enumerate}
\item $\eta(w)=1$; or
\item $\eta(w)\in H_0$; or
\item $\eta(w)\in H_1$.
\end{enumerate}
\end{lemma}
\begin{proof}
The algorithm of the previous lemma determines if $W=\nu(w)$ for some
$w\in F_\infty$, and gives a decomposition 
\[
w=v_0s^{i_0}v_1s^{i_1}\ldots v_ks^{i_k}
\] where each $v_j$ is a reduced word in the free group on
$\{x,c_0,c_1,\dots,c_N\}$, with $N$ bounded by a linear function of $l(w)$. 
Part (1) then follows immediately from
the solution to the word problem in $\G$. 

Writing $r_j=i_0+\ldots+i_{j-1}$ for each $j>0$, with $r_0=0$, we get
\begin{equation}\label{eqn}
w=(s^{r_0}v_0s^{-r_0})(s^{r_1}v_1s^{-r_1})(s^{r_2}v_2s^{-r_2})\ldots (s^{r_k}v_ks^{-r_k})\ .
\end{equation}

Note that $r_j\neq r_{j+1}$. Because $H_0$ is free on the set $S_0$, we have $\redden{\eta(w)}\in H_0$ if and only if \redden{$\eta(v_j)\in \langle c_{r_j-1}x\rangle$ }for each $j$.  \redden{Let $\xi_j$ be the index sum of $x$ in $v_j$.  Then $\eta(v_j)\in \langle c_{r_j-1}x\rangle$ if and only if $\eta(v_j)=(c_{r_j-1}x)^{\xi_j}$, and this equation can be checked using the
solution to the word problem in $\Gamma$.
Thus, the time taken to decide whether or not $w\in H_0$  is bounded by a polynomial  in $\hat{l}(w)$ (using a non-deterministic algorithm if the solution to the word problem in $\G$ was only NP).  This completes our description of the algorithm for part (2).

Finally, $\eta(w)\in H_1$ if and only if $\eta(w)\in H_0$ and $r_j>0$ whenever $v_j$ is non-trivial.  This last condition is easy to check.
\end{proof}

{\em Proof of Theorem \ref{t: Polynomial word problem}}.} 
We can now describe the algorithm to determine whether or not a word $w_0$ in the symbols $c_0,s,t\redden{,x}$ is trivial in $G$.  We shall inductively construct a finite sequence of reduced words $W_0,\ldots,W_N\in \widehat{F}$ with the following properties: $W_n=\nu(w_n)\in \nu(F_n)$; for $n<N$, $\eta(w_n)=1$ if and only if $\eta(w_{n+1})=1$; and $W_N$ is either empty or $\eta(w_N)\neq 1$.   The construction of $w_{n+1}$ from $w_n$ takes at most time $q(\hat{l}(w_n))$, where $q$ is a fixed polynomial (which we may assume is an increasing function on $\N$). And $N$ is bounded by a polynomial in $\hat{l}(w_0)=l(W_0)$.

We proceed as follows. First translate $w_0$ into a word $W_0$ in the symbols $a,b,s,t\redden{,x}$ using the map $\nu$.  A \emph{potential pinch} in any reduced word $W\in \widehat{F}$ is a subword of the form $tXt^{-1}$ or $t^{-1}Xt$, where $X$ is a word in the letters $a,b\redden{,x}$ and $s$.  Given $W_n=\nu(w_n)$, find all potential pinches $tUt^{-1}$ and $t^{-1}Vt$.  If there are no potential pinches but there is at least one occurence of $t$, then
by Britton's lemma $\eta(w_0)\neq 1$ in $G$. If there are no occurences of $t$, we apply the solution to the word problem in $\Gamma*\langle s\rangle\redden{*\langle x\rangle}$
(coming directly from $\G$).  For each potential pinch, check using
Lemma \ref{l2} whether $U=\nu(u)$ with $\eta(u)\in H_0$ or $V=\nu(v)$
with $\eta(v)\in H_1$.  If none of them are, then by Britton's Lemma  we know that $\eta(w_n)\neq 1$, and the algorithm terminates.

On the other hand, suppose that some potential pinch is indeed a pinch, i.e. $\eta(u)\in H_0$ or $\eta(v)\in H_1$.  We shall assume that it is $\redden{\eta(u)}\in H_0$; the case of $\redden{\eta(v)}\in H_1$ is similar.  
As in equation (\ref{eqn}), we write $u$ as a product of subwords $s_n=s^{n+1}c_n\redden{x}s^{-n-1}$ by introducing appropriate cancelling
powers of $s$, and we consider the word $u'$ obtained from $u$ by
replacing each $s_n$  by $s_{n+1}=s^{n+2}c_{n+1}\redden{x}s^{-n-2}$ then
freely reducing. Then $w_{n+1}$ is obtained from $w_n$ by replacing $tut^{-1}$ with $u'$, and $W_{n+1}:=\nu(w_{n+1})$.  Note that the derivation of $W_{n+1}$ from $W_n$ (or the determination that $\eta(w_n)$ is trivial or non-trivial) took at most time $q(\hat{l}(w_n))$, where $q$ is an increasing polynomial.

As $W_{n+1}$ contains fewer potential pinches than $W_n$, this procedure can be iterated at most $l(w_0)$ times, and therefore describes a solution to the word problem in $G$.  It remains to estimate the total running time of the procedure, and to do this we bound the length of $W_n$.    An \emph{$s$-component} of a reduced word $W\in\widehat{F}$ is a maximal subword of the form $s^k$ where $k\neq 0$.  Let $\sigma_n$ be the number of $s$-components of $W_n$.  Note that $l(W_{n+1})=\hat{l}(w_{n+1})\leq \hat{l}(w_n)+4\sigma_n$, that $\sigma_{n+1}\leq\sigma_n$, and that $\sigma_0\leq l(w_0)$.  Therefore, for all $n$, 
\[
l(W_n)=\hat{l}(w_n)\leq (4n+1)l(w_0)\ .
\]
Furthermore, $n$ is bounded above by the number of potential pinches in $w_0$, and hence by $l(w_0)$.  Therefore the running time of the algorithm is bounded above by
\[
\sum_n q(\hat{l}(w_n))\leq \sum_n q((4n+1)l(w_0))\leq l(w_0)q((4l(w_0)+1)l(w_0))
\]
which is polynomial in $l(w_0)$, as required.  This completes the proof of Theorem \ref{t: Polynomial word problem}.\hfill$\square$
\smallskip

\begin{proof}[Proof of Theorem \ref{t:single group}]
It follows immediately from Corollary \ref{c: Matrices} that the group $\SL_\infty(\Z)$ has unsolvable finite presentation problem. Therefore, in order to prove Theorem \ref{t:single group} it remains only to \redden{find a generating set for $\SL_\infty(\Z)$ with respect to which one can solve the word problem in polynomial time (cf.~\cite{Lipton}).

It is well known that $\SL_\infty(\Z)$ is generated by $\{e_{i,j}\mid i,j\in\N,~i\neq j\}$, where $e_{i,j}$ is the elementary matrix with $(i,j)$ entry equal to $1$.  We now define a generating set $\{c_k\mid k\in\N\}$ by
\[
c_k=\left\{
\begin{array}{cl}
e_{i,j} & k=2^i3^j\\
1 & \mathrm{otherwise}\\
\end{array}
\right.~.
\]
Note that $i,j\leq \log_2 k$, so $c_k\in \SL_{\lfloor \log_2 k\rfloor}(\Z)$.  

Given a word $w=c_{k_1}\ldots c_{k_l}$, let $\kappa=\max\{k_p\}$.  In $O(l\log\kappa)$ time, $w$ can be translated into the corresponding 
product of matrices $e_{i_1,j_1}\ldots e_{i_l,j_l}$ in $\SL_{\lfloor \log_2 \kappa\rfloor}(\Z)$.  Multiplication on the right by an elementary matrix corresponds to a single column operation, which involves $O(\log\kappa)$ additions.  Addition of two $d$-digit numbers is of complexity $\Theta(d)$.  By induction, every entry of every matrix is $O(2^l)$, so has $O(l)$ (binary) digits.  It follows that the product $e_{i_1,j_1}\ldots e_{i_l,j_l}$ can be multiplied in $O(l^2 \log\kappa)$ time.

In summary, there is an algorithm that solves the word problem with respect to the generating set $\{c_k\}$ in $O(l^2\log\kappa)$ time, which is polynomial in $\kappa$ and $l$, as required.}
\end{proof}

\section{Groups with solvable finite presentation problem}\label{s: Examples}

In this section we collect classes of groups in which the finite presentation problem is known to be solvable.   We start with classes of coherent groups.  In \cite{stallings_topology_1983}, Stallings explains how his folding technique can be used to compute presentations for subgroups of free groups.

\begin{example}[Free groups]
The finite presentation problem is uniformly solvable in the class of finitely generated free groups.
\end{example}

More generally, there is an algorithm to compute a presentation for any quasiconvex subgroup of a hyperbolic group.  Indeed, I.\ Kapovich described an algorithm that, given a finite presentation $\mathcal{P}$ of a hyperbolic group $\Gamma$, finds a quasiconvexity constant $K$ for any quasiconvex subgroup $\Lambda\subseteq\Gamma$ (and that runs forever if the given subgroup is not quasiconvex) \cite{kapovich_detecting_1996}.  Papasoglu described an algorithm that computes a hyperbolicity constant $\delta$ for $\mathcal{P}$ \cite{papasoglu_algorithm_1996}.  From $K$ and $\delta$ it is easy to compute a hyperbolicity constant for $\Lambda$, and hence a presentation.   One therefore has the following.

\begin{example}[Locally quasiconvex hyperbolic groups]
The finite presentation problem is uniformly solvable in the class of locally quasiconvex hyperbolic groups.
\end{example}

In a context where one understands the non-quasiconvex subgroups, it may not be necessary to assume local quasiconvexity.  This is the case with the fundamental groups of closed hyperbolic 3-manifolds.

\begin{example}[Closed hyperbolic 3-manifolds]
The finite presentation problem is uniformly solvable in the class of fundamental groups of closed hyperbolic 3-manifolds.  Indeed, Canary showed \cite{canary_mardens_2008} that a consequence of Marden's Tameness Conjecture (proved by Agol \cite{agol_tameness_2004} and, independently, Calegari and Gabai \cite{calegari_shrinkwrapping_2006}) is that every finitely generated, geometrically infinite (i.e. non-quasiconvex) subgroup of a closed hyperbolic 3-manifold group $\Gamma$ is commensurable with a fibre in a finite-sheeted cover that fibres over the circle.  Therefore, the finite presentation problem for closed hyperbolic 3-manifolds can be solved by two algorithms running in parallel: one is Kapovich's algorithm
seeking a quasiconvexity constant; the other enumerates finite-index subgroups of $\Gamma$ and seeks to identify the input subgroup as a virtual fibre.
\end{example}

Recall that a group $\Gamma$ has the \emph{local retractions property} if for every finitely generated subgroup $\Lambda$ there is a subgroup $K$ of finite index in $\Gamma$ such that $K$ contains $\Lambda$ and the inclusion map $\Lambda\hookrightarrow K$ has a left inverse (called a \emph{retraction}).  Given a retraction  $K\to \Lambda$ and a presentation for $K$, it is easy to write down a presentation for $\Lambda$. Groups that enjoy the local retractions property include limit groups \cite{wilton_halls_2008}.

\begin{example}[Limit groups]
The finite presentation problem is uniformly solvable in the class of limit groups \cite{groves_enumerating_2009}.
\end{example}

Kapovich, Weidmann and Myasnikov adapted Stallings's folding machinery to work in more general graphs of groups \cite{kapovich_foldings_2005}.  Their techniques apply in particular to many right-angled Artin groups.

\begin{example}[Coherent right-angled Artin groups]
Coherent right-angled Artin groups have a solvable finite presentation problem.
\end{example}

Our results show that this statement cannot be improved to cover all right-angled Artin groups.

McCammond and Wise generalised Stallings's ideas in a different direction to construct many examples of coherent groups \cite{mccammond_coherence_2005}.  Their procedure often yields an algorithm for computing presentations, and has applications to small-cancellation groups, one-relator groups and Coxeter groups.  See also Payne and Rees \cite{payne_computing_2006} and Schupp \cite{schupp_coxeter_2003}.

The examples listed so far are all coherent.  In particular, they have solvable finite presentability problem.  Remarkably, there are numerous examples of incoherent groups with unsolvable finite presentability problem but solvable finite presentation problem. The first example is the direct product of two non-abelian free groups.

\begin{example}[$F\times F$]
Let $F$ be a finitely generated free group.  Applying the fibre product construction to an undecidable sequence such as the Collins--Miller groups, and observing that the fibre product is finitely presented if and only if the quotient group is finite \cite{grunewald_groups_1978}, we see that the finite presentability problem is unsolvable in $F\times F$.  However, it follows from the work of Baumslag and Roseblade \cite{baumslag_subgroups_1984} that the finite presentation problem is solvable in $F\times F$.
\end{example}

A general method for solving the finite presentation problem comes from the work of Bridson, Howie, Miller and Short \cite{bridson_finitely_2008}. They establish an algorithm that, given finite presentations for groups $G_1,\ldots,G_n$ and a finite subset $S\subseteq G_1\times\ldots\times G_n$ will construct a finite presentation for $\Lambda=\langle S\rangle$ if the projection of $\Lambda$  to each pair of factors $G_i\times G_j$ is of finite index.  In certain circumstances this leads to a solution to the finite presentation problem.  For example, if one combines this theorem with the algorithm for embedding a residually free group into a direct product of limit groups \cite{bridson_finitely_2008}, then the characterisation of finitely presented subgroups in such direct products \cite{bridson_subgroups_2010} can be used to prove the following statements.

\begin{example}[Residually free groups \cite{bridson_finitely_2008}]
The finite presentation problem is uniformly solvable in the class of finitely presented residually free groups.  On the other hand, the following are equivalent for a finitely presented residually free group $\Gamma$:
\begin{enumerate}
\item the finite presentability problem is solvable in $\Gamma$;
\item $\Gamma$ does not contain $F\times F$;
\item $\Gamma$ is either a limit group or a direct product of a limit group and a free abelian group;
\item $\Gamma$ is coherent.
\end{enumerate}
\end{example}

Theorem \ref{t:main} shows that, although the finite presentation problem is uniformly solvable in the class of residually free groups, it is not uniformly solvable in the class of residually (2-dimensional) hyperbolic groups.

\section{Questions} \label{s:qus}

The proof of Theorem \ref{t:single group} gave very little information about the constructed group, other than that it has polynomial Dehn function.

\begin{question}
How might one improve the properties of the group in Theorem \ref{t:single group}?  For instance, is there a direct product of two hyperbolic groups with unsolvable finite presentation problem?  Is the finite presentation problem unsolvable in $\SL_n(\Z)$ for some finite $n$?
\end{question} 

As observed in the previous section, most known solutions of the finite presentation problem arise from proofs of coherence.  

\begin{question}
Is there a class of coherent groups in which the uniform finite presentation problem is unsolvable?
\end{question}

Famously, Serre asked if $\SL_3(\Z)$ is coherent (see page 734 of \cite{serre_problem_1974}).

\begin{question}
Are the finite presentation and finite presentability problems solvable in $\SL_3(\Z)$?
\end{question}

The following question was posed in \cite{groves_conjugacy_????}.

\begin{question}
Is the finite presentation problem uniformly solvable in the class of word-hyperbolic groups?
\end{question}

\bibliographystyle{plain}

\begin{thebibliography}{10}
\bibitem{agol_tameness_2004}
Ian Agol.
\newblock Tameness of hyperbolic 3-manifolds.
\newblock {\em arXiv:math/0405568}, May 2004.

\bibitem{baumslag_fibre_2000}
Gilbert Baumslag, Martin~R. Bridson, Charles F.~Miller {III}, and Hamish Short.
\newblock Fibre products, non-positive curvature, and decision problems.
\newblock {\em Commentarii Mathematici Helvetici}, 75(3):457---477, 2000.


\bibitem{bms}
G. Baumslag, C.F.~Miller {III}, and H. Short.
\newblock Unsolvable problems about small cancellation and word hyperbolic groups.
\newblock {\em Bulletin of the London Math. Soc.}, 26:97---101, 1994.


\bibitem{baumslag_subgroups_1984}
Gilbert Baumslag and James~E. Roseblade.
\newblock Subgroups of direct products of free groups.
\newblock {\em Journal of the London Mathematical Society. Second Series},
 30(1):44--52, 1984.

\bibitem{birget_isoperimetric_2002}
{Jean-Camille} Birget, A.~Yu. Ol'shanskii, Eliyahu Rips, and Mark~V. Sapir.
\newblock Isoperimetric functions of groups and computational complexity of the
 word problem.
\newblock {\em Annals of Mathematics. Second Series}, 156(2):467---518, 2002.

\bibitem{bogley_homological_2002}
{W.A.} Bogley and J.~Harlander.
\newblock Homological decision problems for finitely generated groups with
 solvable word problem.
\newblock {\em Internat. J. Algebra Comput.}, 12:213---221, 2002.

\bibitem{bridson_decision_2010}
Martin~R. Bridson.
\newblock Decision problems and profinite completions of groups.
\newblock {\em J. Algebra}, to appear, 2010.

\bibitem{bridson_metric_1999}
Martin~R. Bridson and Andr\'e Haefliger.
\newblock {\em Metric spaces of non-positive curvature}, volume 319 of {\em
 Grundlehren der Mathematischen Wissenschaften {[Fundamental} Principles of
 Mathematical Sciences]}.
\newblock {Springer-Verlag}, Berlin, 1999.

\bibitem{bridson_finitely_2008}
Martin~R. Bridson, James Howie, Charles F.~Miller {III}, and Hamish Short.
\newblock On the finite presentation of subdirect products and the nature of residually free groups. Preprint 2008.
\newblock {\em {arXiv:0809.3704v2}}.

\bibitem{bridson_subgroups_2010}
Martin~R. Bridson, James Howie, Charles F.~Miller {III}, and Hamish Short.
\newblock Subgroups of direct products of limit groups.
\newblock {\em Annals of Mathematics. Second Series}, 170:1447--1467, 2010.

\bibitem{BMiller-decide} 
Martin~R. Bridson and Charles F.~Miller {III}. 
\newblock   Recognition of subgroups of direct products of
hyperbolic groups.
\newblock {\em  Procedings of the American Mathematical Society}, 132:  59--65, 2003.

\bibitem{BMiller10}
Martin~R. Bridson and Charles F.~Miller {III}.
\newblock Structure and finiteness properties of subdirect products of groups.
\newblock {\em  Procedings of the London Mathematical Society. Third series.}  98(3):631--651, 2009.


\bibitem{brown_cohomology_1994}
Kenneth~S. Brown.
\newblock {\em Cohomology of groups}, volume~87 of {\em Graduate Texts in
 Mathematics}.
\newblock {Springer-Verlag}, New York, 1994.
\newblock Corrected reprint of the 1982 original.

\bibitem{calegari_shrinkwrapping_2006}
Danny Calegari and David Gabai.
\newblock Shrinkwrapping and the taming of hyperbolic 3-manifolds.
\newblock {\em Journal of the American Mathematical Society}, 19(2):385446
 (electronic), 2006.

\bibitem{canary_mardens_2008}
Richard~D. Canary.
\newblock Marden's {{T}ameness} {{C}onjecture:} history and applications.
\newblock In {\em Geometry, Analysis and Topology of Discrete groups, ed. by L.
 Ji, K. Liu, L. Yang and {S.T.} Yau}, pages 137---162. Higher Education
 Press, 2008.

\bibitem{clapham_embedding_1967}
C.~R.~J. Clapham.
\newblock An embedding theorem for finitely generated groups.
\newblock {\em Proceedings of the London Mathematical Society. Third Series},
 17:419---430, 1967.


\bibitem{collins}
D.\ J.\ Collins.
\newblock The word, power and order problems in finitely presented groups.
\newblock In {\em Word Problems, ed. by Boone, Cannonito, and Lyndon}, pages 401---420. North-Holland, Amsterdam, 1973.


\bibitem{collins_word_1999}
Donald~J. Collins and Charles F.~Miller {III}.
\newblock The word problem in groups of cohomological dimension 2.
\newblock In {\em Groups St. Andrews 1997 in Bath, I}, volume 260 of {\em
 London Math. Soc. Lecture Note Ser.}, pages 211---218. Cambridge Univ.\ Press,
 Cambridge, 1999.

\bibitem{gordon_embedding_1995}
C.~{McA.} Gordon.
\newblock Some embedding theorems and undecidability questions for groups.
\newblock In {\em Combinatorial and geometric group theory {(Edinburgh,}
 1993)}, volume 204 of {\em London Math.\ Soc.\ Lecture Note Ser.}, pages
 105--110. Cambridge Univ. Press, Cambridge, 1995.

\bibitem{groves_conjugacy_????}
Daniel Groves and Henry Wilton.
\newblock Conjugacy classes of solutions to equations and inequations over
 hyperbolic groups.
\newblock {\em Journal of Topology},  3:311--332, 2010.

\bibitem{groves_enumerating_2009}
Daniel Groves and Henry Wilton.
\newblock Enumerating limit groups.
\newblock {\em Groups, Geometry and Dynamics}, 3(3):389---399, 2009.

\bibitem{grunewald_groups_1978}
F.~Grunewald.
\newblock On some groups which cannot be finitely presented.
\newblock {\em Jour. London Math. Soc.}, 17:427---436, 1978.

\bibitem{haglund_special_2008}
Fr\'ed\'eric Haglund and Daniel~T. Wise.
\newblock Special cube complexes.
\newblock {\em Geometric and Functional Analysis}, 17(5):1551--1620, 2008.

\bibitem{higman_embedding_1949}
Graham Higman, B.~H. Neumann, and Hanna Neumann.
\newblock Embedding theorems for groups.
\newblock {\em Journal of the London Mathematical Society. Second Series},
 24:247254, 1949.

\bibitem{humphreys_reflection_1990}
James~E. Humphreys.
\newblock {\em Reflection groups and Coxeter groups}, volume~29 of {\em
 Cambridge Studies in Advanced Mathematics}.
\newblock Cambridge University Press, Cambridge, 1990.

\bibitem{kapovich_detecting_1996}
Ilya Kapovich.
\newblock Detecting quasiconvexity: algorithmic aspects.
\newblock In {\em Geometric and computational perspectives on infinite groups
 {(Minneapolis,} {MN} and New Brunswick, {NJ,} 1994)}, volume~25 of {\em
 {DIMACS} Ser. Discrete Math. Theoret. Comput. Sci.}, page 9199. Amer.
 Math. Soc., Providence, {RI}, 1996.

\bibitem{kapovich_foldings_2005}
Ilya Kapovich, Richard Weidmann, and Alexei Miasnikov.
\newblock Foldings, graphs of groups and the membership problem.
\newblock {\em Internat. J. Algebra Comput.}, 15(1):95---128, 2005.

\bibitem{Lipton}
Richard J. Lipton and 
Y. Zalcstein.
\newblock 
Word Problems Solvable in Logspace.
\newblock{\em J. Assoc. Computing Machinery}, 24(3):522--526, 1977.

\bibitem{mccammond_coherence_2005}
J.~P. {McCammond} and D.~T. Wise.
\newblock Coherence, local quasiconvexity, and the perimeter of 2-complexes.
\newblock {\em Geometric and Functional Analysis}, 15(4):859---927, 2005.

\bibitem{mccool}
James McCool.
\newblock The order problem and the power problem for free product sixth-groups.
\newblock {\em Glasgow Math. J.}, 10: 1--9, 1969. 

\bibitem{cfm-thesis}  C.~F.~Miller~III.
\newblock{\em On group-theoretic decision problems and their classification}.
\newblock Annals of Mathematics Studies, No. 68.
\newblock Princeton University Press, Princeton, {N.J.}, 1971.

\bibitem{milnor_introduction_1971}
John Milnor.
\newblock {\em Introduction to algebraic K-theory}.
\newblock Annals of Mathematics Studies, No. 72.
\newblock Princeton University Press, Princeton, {N.J.}, 1971.

\bibitem{papasoglu_algorithm_1996}
P.~Papasoglu.
\newblock An algorithm detecting hyperbolicity.
\newblock In {\em Geometric and computational perspectives on infinite groups
 {(Minneapolis,} {MN} and New Brunswick, {NJ,} 1994)}, volume~25 of {\em
 {DIMACS} Ser. Discrete Math. Theoret. Comput. Sci.}, pages 193---200. Amer.
 Math. Soc., Providence, {RI}, 1996.

\bibitem{payne_computing_2006}
Oliver Payne and Sarah Rees.
\newblock Computing subgroup presentations, using the coherence arguments of
 {McCammond} and Wise.
\newblock {\em Journal of Algebra}, 300(1):109---133, 2006.

\bibitem{rips_subgroups_1982}
E.~Rips.
\newblock Subgroups of small cancellation groups.
\newblock {\em The Bulletin of the London Mathematical Society}, 14(1):45---47,
 1982.

\bibitem{sapir_isoperimetric_2002}
Mark~V. Sapir, {Jean-Camille} Birget, and Eliyahu Rips.
\newblock Isoperimetric and isodiametric functions of groups.
\newblock {\em Annals of Mathematics. Second Series}, 156(2):345466, 2002.

\bibitem{schupp_coxeter_2003}
Paul~E. Schupp.
\newblock Coxeter groups, 2-completion, perimeter reduction and subgroup
 separability.
\newblock {\em Geometriae Dedicata}, 96:179---198, 2003.

\bibitem{serre_problem_1974}
{Jean-Pierre} Serre.
\newblock Problem section.
\newblock In {\em Proceedings of the Second International Conference on the
 Theory of Groups {(Canberra,} 1973)}, volume 372 of {\em Lecture Notes in
 Mathematics}. {Springer-Verlag}, 1974.

\bibitem{stallings_topology_1983}
John~R. Stallings.
\newblock Topology of finite graphs.
\newblock {\em Inventiones Mathematicae}, 71(3):551--565, 1983.

\bibitem{wilton_halls_2008}
Henry Wilton.
\newblock Hall's {{T}heorem} for limit groups.
\newblock {\em Geometric and Functional Analysis}, 18(1):271--303, 2008.

\bibitem{wise_residually_2003}
Daniel~T. Wise.
\newblock A residually finite version of {{R}ips's} construction.
\newblock {\em The Bulletin of the London Mathematical Society}, 35(1):23---29,
 2003.

\end{thebibliography}

\end{document}